\documentclass[11pt]{article}
\usepackage{amsmath, amssymb, amsfonts, amsthm, graphicx, relsize, mathtools, mathrsfs, euscript, bbm, bbold, float, hyperref, enumitem, url, breakurl, stmaryrd, xcolor}
\usepackage[english]{babel}
\usepackage{titlesec}
\titlelabel{\thetitle.\,\,}

\hypersetup{
colorlinks,
linkcolor={black!50!black},
citecolor={black!50!black},
urlcolor={black!80!black}
}

\allowdisplaybreaks

\makeatletter
\newtheorem*{rep@theorem}{\rep@title}
\newcommand{\newreptheorem}[2]{%
\newenvironment{rep#1}[1]{%
 \def\rep@title{#2 \ref{##1}}%
 \begin{rep@theorem}}%
 {\end{rep@theorem}}}
\makeatother

\theoremstyle{plain}
\newtheorem{theorem}{Theorem}[section]
\newreptheorem{theorem}{Theorem}
\newtheorem{lemma}[theorem]{Lemma}

\newtheorem{proposition}[theorem]{Proposition}

\theoremstyle{definition}
\newtheorem{definition}[theorem]{Definition}
\newtheorem{remark}[theorem]{Remark}

\makeatletter
\def\blfootnote{\gdef\@thefnmark{}\@footnotetext}
\makeatother

\DeclareMathAlphabet{\mathbbmsl}{U}{bbm}{m}{sl}

\begin{document}

\title{\bf{\Large Bootstrap percolation on the Hamming graphs}\\[9mm]}

\author{
Meysam Miralaei$^{^{1}}$  \qquad  Ali  Mohammadian$^{^{2}}$  \qquad   Behruz  Tayfeh-Rezaie$^{^{2}}$\\[4mm]
School of Mathematics,\\
Institute for Research in Fundamental Sciences (IPM),\\
P.O. Box 19395-5746, Tehran, Iran\\[3mm]
\href{mailto:m.miralaei@ipm.ir}{m.miralaei@ipm.ir} \qquad
\href{mailto:ali\_m@ipm.ir}{ali\_m@ipm.ir}  \qquad
\href{mailto:tayfeh-r@ipm.ir}{tayfeh-r@ipm.ir}\vspace{9mm}}

\blfootnote{\hspace*{-6mm}$^{^1}$Partially  supported by a grant from IPM.\\
$^{^2}$Partially  supported by Iran  National  Science Foundation   under project number  99003814.}

\date{}

\maketitle

\begin{abstract}
\noindent  The $r$-edge bootstrap percolation on a graph     is an  activation process of the  edges. The process starts with   some   initially activated edges   and  then, in each round,    any inactive edge whose one of   endpoints is incident to       at least $r$ active edges  becomes activated.
A  set of  initially activated edges   leading   to the activation of all  edges  is said to be a  percolating set.
Denote the  minimum  size of a  percolating set in the  $r$-edge  bootstrap percolation process  on a graph $G$  by $m_e(G, r)$.
The importance of the  $r$-edge bootstrap percolation relies on the fact that  $m_e(G,  r)$
provides bounds on
$m(G,  r)$, that is,   the minimum size of a percolating set in the $r$-neighbor
bootstrap percolation process on $G$.
In this paper, we  explicitly determine
$m_e(K_n^d, r)$, where $K_n^d$  is   the Cartesian product of $d$ copies of the complete graph  on $n$ vertices  which is referred   as       Hamming graph.
Using this, we show that  $m(K_n^d, r)=(1+o(1))\frac{d^{r-1}}{r!}$  when $n, r$ are fixed and $d$  goes to infinity which  extends a   known  result  on hypercubes. \\[3mm]
\end{abstract}

\section{Introduction}

Bootstrap percolation processes on graphs     can be interpreted as  a family of   cellular automata,    a concept     introduced  in 1966   by von Neumann \cite{von}.
They   have  been extensively investigated    in several diverse fields such as  combinatorics,  probability theory, statistical physics  and social sciences.
The      $r$-neighbor bootstrap percolation      is the most studied    of  such  processes    which   was  firstly introduced  in 1979  by Chalupa, Leath  and Reich \cite{cha}.
This  process has  also been treated  in the literature under  other    names like  irreversible threshold, influence propagation   and dynamic monopoly.

Throughout   this  paper, all graphs    are assumed to be finite,    undirected,     without loops and  multiple edges.
For   a graph $G$,   the vertex set of   $G$  is denoted by $V(G)$ and  the edge set of $G$ is denoted by $E(G)$.
The    $r$-neighbor bootstrap percolation  can be defined formally as follows.
Given a nonnegative    integer $r$ and a graph $G$, the  {\sl $r$-neighbor bootstrap percolation process} on $G$  begins
with a subset $V_0$ of  initially activated vertices of $G$
and then,   at    step $i$ of the process,    the set $V_i$ of active  vertices    is
$$V_i=V_{i-1}\cup\left\{
v\in V(G) \, \left| \,
\begin{array}{ll}
    \text{The vertex $v$ is adjacent to} \\
    \text{at least $r$ vertices in $V_{i-1}$.}
   \end{array}
    \right.
\right\}$$
for each  $i\geqslant1$.
We say   $V_0$    is a {\sl percolating set} of $G$   if $V_t=V(G)$ for some $t\geqslant0$.
An  extremal problem here  is to determine the minimum size of a percolating set which  is denoted by $m(G,r)$.
The size of percolating sets has been studied for      various families of  graphs such as  hypercubes \cite{mor},   grids \cite{n=2},       trees \cite{rie} and random graphs \cite{fei}.

An  edge version    of   the $r$-neighbor bootstrap percolation can be defined  by considering a special case of the so-called
`graph bootstrap percolation'. The concept of  graph bootstrap
percolation     was firstly  introduced in 1968   by Bollob\'as      under a different name  \cite{bol}  and was later studied
in 2012  by  Balogh,     Bollob\'as     and   Morris   under the current name \cite{bal}.
Graph  bootstrap percolation  can be defined formally as follows.
Given two  graphs $G$ and $H$, the {\sl  $H$-bootstrap percolation process} on $G$  begins
with a subset $E_0$ of initially activated edges of $G$
and then,   at   step $i$ of the process,    the set $E_i$  of active  edges   is
$$E_i=E_{i-1}\cup\left\{
e\in E(G) \, \left| \,
\begin{array}{ll}
    \text{There exists  a subgraph $H_e$ of $G$ such} \\
    \text{that $H_e$ is isomorphic to $H$, $e\in E(H_e)$} \\
    \text{and $E(H_e)\setminus\{e\}\subseteq E_{i-1}$.}
   \end{array}
    \right.
\right\}$$
for each    $i\geqslant1$. The  set $E_0$  is called    a {\sl percolating set} of  $G$  if   $E_t=E(G)$ for some $t\geqslant0$.
The minimum size of a percolating  set in the $H$-bootstrap percolation  process  on $G$  is  equal to the so-called       {\sl weak saturation number} of $H$ in $G$ and is   denoted by $\mathrm{wsat}(G, H)$.
Denoting the star graph   on    $m$ edges by $S_m$,
we refer to the  $S_{r+1}$-bootstrap percolation as  the {\sl  $r$-edge bootstrap percolation} which can be considered   as an edge    analogue   of   the $r$-neighbor bootstrap percolation.
For simplicity, we write
$m_e(G, r)$ instead of $\mathrm{wsat}(G, S_{r+1})$.
The    $2$-edge bootstrap percolation  had been studied in   1984  by Lenormand and Zarcone  under a different name  \cite{len}.
By a result from  \cite{ham}, we have
\begin{equation}\label{m-me}\frac{m_e(G, r)}{r}\leqslant m(G, r)\leqslant m_e(G, r)+\Big|\big\{v\in V(G) \, \big| \, \deg(v)<r\big\}\Big|,\end{equation}
where $\deg(v)=|\{x\in V(G) \, | \, x \text{   is adjacent to } v\}|$.

Let us fix here some notation and terminology used in the rest of the paper.
For every two adjacent  vertices $u$ and $v$,   the edge joining    $u$ and $v$ is denoted  by $uv$ or $\{u, v\}$.
The {\sl Cartesian product} of two graphs $G$ and $H$, denoted by $G\square H$, is the graph with vertex set $V(G)\times V(H)$  in which two vertices $(g_1, h_1)$ and $(g_2, h_2)$ are adjacent if and only if either $g_1=g_2$ and $h_1h_2\in E(H)$  or  $g_1g_2\in E(G)$ and $h_1=h_2$.
Denote by  $G^d$     the  Cartesian product of $d$ vertex disjoint  copies of the    graph $G$.
For any   integer $n$, we let  $ \llbracket n \rrbracket =\{0, 1, \ldots, n-1\}$ if $n\geqslant 1$ and      $ \llbracket n \rrbracket =\varnothing$ otherwise.
We denote the    complete graph on $n$ vertices by  $K_n$ and we   always consider    $ \llbracket n \rrbracket =\{0, 1, \ldots, n-1\}$  as the vertex set of $K_n$.
The graph $K_n^d$       is called  a  {\sl  Hamming graph}  of dimension $d$.

Balister,    Bollob\'{a}s,    Lee   and   Narayanan   \cite{balister}   gave    the  lower bound  $(r/d)^d$   and the    upper bound  $r^d/(2d!)$ on $m(K_n^d, r)$ for  $n\geqslant2$.
This  along with \eqref{m-me}  motivates   to study    $m_e(K_n^d, r)$.
In the current  paper, we apply  the  polynomial technique  introduced by Hambardzumyan,   Hatami and  Qian  \cite{ham}    to  get an explicit formula  for  $m_e(K_n^d, r)$ for all values of  $n, r, d$.  Note that $m_e(G, r)=|E(G)|$ if  $r\geqslant \max\{\deg(v)   \, | \, v\in V(G)\}$. In particular, $m_e(K_n^d, r)=\frac{(n-1)d}{2}n^d$ if $r\geqslant (n-1)d$.

\begin{theorem}\label{sigmaT2}
Let  $n\geqslant 2$  and $d, r\geqslant 0$ be three integers with  $ 0\leqslant r\leqslant (n-1)d$. Then,
\begin{align*}
	m_e\left(K_n^d, r\right) = \sum_{i_1=0}^{r-1} \sum_{i_2=0}^{r-i_1-1} \cdots \sum_{i_{n-1}=0}^{r-i_1-\cdots-i_{n-2}-1}
(r-i_1-\cdots-i_{n-1})\binom{d}{i_1}\binom{i_1}{i_2}\cdots \binom{i_{n-2}}{i_{n-1}}.
\end{align*}
\end{theorem}
\noindent Bidgoli,   Mohammadian  and    Tayfeh-Rezaie \cite{bid}  had proved that $$m_e(K_n^d, r) = {{d+r}\choose{d+1}}$$  if     $0 \leqslant  r\leqslant n-1$. In the current  paper, we particularly show that
\begin{align*}
m_e\left(K_n^d, r\right) = {{nd-r}\choose{d+1}}+\left(r-\frac{(n-1)d}{2}\right)n^d
\end{align*}
if $(n-1)(d-1)\leqslant r\leqslant (n-1)d$.

An    explicit formula  for  $m_e(K_2^d, r)$  was   already found by Morrison and   Noel    \cite{mor}
which gives the  asymptotic formula $m(K_2^d, r)=\frac{1 + o(1)}{r}{{d}\choose{r-1}}$ when $r$ is fixed and $d$ goes to infinity, settling  a conjecture raised by  Balogh  and   Bollob\'as    \cite{balbol}.
We generalize  the  asymptotic  result   by   establishing the following theorem as a consequence of  Theorem \ref{sigmaT2}.

\begin{theorem}\label{assxxx}
Let $n\geqslant2$ and $r$ be two fixed positive  integers  and let  $d$ be an   integer  tending  to infinity. Then,
\[
  m\left(K_n^d, r\right)=\big(1+o(1)\big)\frac{d^{r-1}}{r!}.
\]
\end{theorem}
\noindent Another asymptotic result, which  was proved by  Bidgoli,    Mohammadian  and   Tayfeh-Rezaie    \cite{bid},    states that    $m(K_n^d, r)=\frac{1+o(1)}{(d+1)!}r^d$    when both $r, d$ go to infinity with $d=o(\sqrt{r})$ and $n\geqslant r+1$.
Furthermore,   recursive formulas  for  $m_e(P_n^d, r)$ and  $m_e(C_n^d, r)$ are found in  \cite{ham}, where $P_n$ is  the path graph on $n$ vertices and  $C_n$ is  the cycle  graph on $n$ vertices.

The rest  of the paper is organized as follows. In Section \ref{poly}, we recall the  polynomial technique     which is used      to get  a lower bound on  $m_e(G, r)$ and we present a new proof for it.       In Section \ref{m_eK_n^d}, we present  an  explicit  formula  for  $m_e(K_n^d, r)$ for all values of  $n, r, d$.
Using this, we present  an asymptotic formula for  $m(K_n^d, r)$   when $n, r$ are fixed and $d$ tends to infinity.

\section{The algebraic method}\label{poly}

In this section, we recall  a polynomial technique  which is  introduced by Hambardzumyan,   Hatami and  Qian  \cite{ham}  and  we will    use it      to get  a lower bound on  $m_e(K_n^d, r)$ in the next section. We    show  here  that the   polynomial technique can be regarded as a special case of   a general framework      due to Balogh, Bollob\'{a}s, Morris and Riordan \cite{linear}. A short  proof of the following interesting  lemma  is given   in \cite{Kronenberg}.

\begin{theorem}[{\cite{linear}}]\label{BBMR}
Let $ G,  F $ be two graphs and let
$ W $ be an arbitrary  vector space. Assume  that there is    a subset $ \{w_e \, | \, e \in E(G)\} $ of $W$ such that
for each  copy $ F' $ of $F$ in $G$ there are nonzero scalars $ \{\lambda_{e,F'} \, | \, e \in E(F')\} $ such that
$\sum_{e\in E(F')} \lambda_{e, F'}w_e=0$. Then,
$$\mathrm{wsat}(G, F) \geqslant \dim\Big(\mathrm{span}\big\{w_e \, \big| \, e \in E(G)\big\}\Big).$$
\end{theorem}

We note that   the following definition  is  slightly different  from the  original version.

\begin{definition}[{\cite{ham}}]\label{defhh}
Let $r$ be a nonnegative   integer and   let   $G$  be a graph equipped  with a  proper edge coloring   $c : E(G)\rightarrow\mathbbmsl{R}$. Let $W_c(G, r)$ be the vector space over $\mathbbmsl{R}$ consisting of all
functions $\varphi : E(G)\rightarrow\mathbbmsl{R}$ for   which there exist  polynomials $\{P_v(x)\}_{v\in V(G)}$ satisfying
\begin{itemize}
\item[{\rm (i)}] $\deg P_v(x)\leqslant r-1$ for any vertex  $v\in V(G)$;
\item[{\rm (ii)}]  $P_u(c(uv))=P_v(c(uv))=\varphi(uv)$ for each edge $uv\in E(G)$.
\end{itemize}
It is said that the polynomials  $\big\{P_v(x)\big\}_{v\in V(G)}$ {\sl recognize} $\varphi$. Notice that we adopt the convention that the degree of the zero polynomial is   $-1$.
\end{definition}

The next theorem  provides an interesting  linear algebraic lower bound on $m_e(G, r)$.
Some other nice  applications of vector spaces and polynomials    for bootstrap percolation processes on graphs  can be found in \cite{balister, linear, bid, Kronenberg, ham, mor}.   Indeed,    the literature is full of proofs via linear and multilinear  algebraic techniques for which no combinatorial proof is known.

\begin{theorem}[{\cite{ham}}]\label{hh}
Let  $r$ be a nonnegative   integer  and let $c : E(G) \rightarrow\mathbbmsl{R}$ be a proper edge coloring of a graph $G$. Then,  $m_e(G, r)\geqslant\dim W_c(G, r)$.
\end{theorem}

\begin{proposition}
Theorem \ref{hh} can be proved by Theorem \ref{BBMR}.
\end{proposition}

\begin{proof}
Let $\{\varphi_1, \ldots, \varphi_k  \}$ be  a basis for $W_c(G,r)$.
We assign to every edge $e\in E(G)$ a vector $w_e=(\varphi_1(e), \ldots, \varphi_k(e))$ in $\mathbbmsl{R}^k$.
We claim that for every copy of $S_{r+1}$ in $G$ with the central vertex $v$ and the edge set $\{e_1, \ldots, e_{r+1}\}$ there are nonzero scalars $\lambda_{1}, \ldots,  \lambda_{{r+1}} $ such that
$\sum_{i=1}^{r+1} \lambda_{i}w_{e_i}=0$. We know that there are nonzero scalers $\lambda_{1}, \ldots,  \lambda_{{r+1}} $ such that
\[
\left[
\begin{array}{ccccc}
	1 &  1  &    \cdots & 1  \\
	c_1  & c_2  &  \cdots  &c_{r+1}  \\
	c_1^2  & c_2^2  &  \cdots  &c_{r+1}^2 \\
	\vdots & \vdots &   \vdots & \vdots  \\
	c_1^{r-1}  & c_2^{r-1}  &  \cdots  &c_{r+1}^{r-1} \\
\end{array}\right]
\left[
\begin{array}{c}
\lambda_{1} \\  \lambda_{2} \\    \vdots \\ \lambda_{{r+1}}
\end{array}\right] =0,
\]
where $ c_i =c(e_i)$.
Therefore, for any polynomial $P(x)$ of degree $r-1$, we have
\[
\left[
\begin{array}{ccccc}
P(c_1) &  P(c_2)  &    \cdots & P(c_{r+1})
\end{array}\right]
\left[
\begin{array}{c}
\lambda_{1} \\  \lambda_{2} \\    \vdots \\ \lambda_{{r+1}}
\end{array}\right] =0
\]
and so
\begin{align} \label{mtx}
\left[
\begin{array}{ccccc}
P_v^{(1)}(c_1) &  P_v^{(1)}(c_2)  &    \cdots & P_v^{(1)}(c_{r+1}) \\
\vdots & \vdots  &    \vdots & \vdots \\
P_v^{(r+1)}(c_1) &  P_v^{(r+1)}(c_2)  &    \cdots & P_v^{(r+1)}(c_{r+1})
\end{array}\right]
\left[
\begin{array}{c}
\lambda_{1} \\  \lambda_{2} \\    \vdots \\ \lambda_{{r+1}}
\end{array}\right] =0,
\end{align}
where  $P_v^{(i)}$ is  the polynomial corresponding to the vertex $v$ and the function $\varphi_i$ for $i=1, \ldots, r+1$.
Equality \eqref{mtx} is equivalent to $\sum_{i=1}^{r+1} \lambda_{i}w_{e_i}=0$, as claimed.
Now, using Theorem \ref{BBMR}, we have
\begin{equation*}m_e(G, r)= \mathrm{wsat}(G, S_{r+1}) \geqslant \dim\Big(\mathrm{span}\big\{w_e \, \big| \, e \in E(G)\big\}\Big)=k.\qedhere\end{equation*}
\end{proof}

\section{Percolating sets of Hamming graphs}\label{m_eK_n^d}

In this section, we first  present   an  explicit  formula  for  $m_e(K_n^d, r)$ for all positive integers  $n, r, d$ and then,
as a consequence,  we  give     an asymptotic formula for  $m(K_n^d, r)$   when $n, r$ are fixed and $d$ tends to infinity.
Throughout this section, for every  integers $n, r, d$, we let  $f=r-1-(n-1)(d-1)$ and
\[
g=\left\{\begin{array}{ll}
1  &   \quad   \mbox{ if } f\leqslant -2,\\
\vspace{-1mm}\\
f+2  &   \quad  \mbox{ if } -1\leqslant f\leqslant n-2,\\
\vspace{-1mm}\\
n  & \quad  \mbox{ if }  f\geqslant n-1.
\end{array}\right.
\]
Throughout this section,   $m_e(G, i)$ and $W_c(G, i)$ are respectively  interpreted as  $0$  and   $\{0\}$  for any  graph $G$ if  $i<0$.

\begin{lemma}\label{umeg}
	Let $n, r, d$ be   three  positive integers.
	Then,
	\begin{align}\label{sigma-me}
	m_e\left(K_n^d, r\right)\leqslant\sum_{i=0}^{n-1} m_e\left(K_n^{d-1}, r-i\right) + \binom{g}{2}n^{d-1}.
	\end{align}
\end{lemma}

\begin{proof}
	If $ r\geqslant(n-1)d$, then $g=n$ and
	\[
	m_e\left(K_n^{d}, r\right)=\left|E\left(K_n^{d}\right)\right|=\sum_{i=0}^{n-1}\left|E\left(K_n^{d-1}\right)\right|+\binom{n}{2}n^{d-1}=\sum_{i=0}^{n-1} m_e\left(K_n^{d-1}, r-i\right) + \binom{g}{2}n^{d-1}.
	\]
	So, assume that $r< (n-1)d$.
	For  any $i\in \llbracket n  \rrbracket $, consider the subgraph $H_i$  of $ K_n^d$ induced on  $\{(v, i)\in V(K_n^d) \,     |  \,  v\in V(K_n^{d-1})\}$ which  is clearly  isomorphic to $K_n^{d-1}$.  Also, consider   a  percolating  set $E_i$  of the  minimum possible size in
	the  $(r-i)$-edge bootstrap percolation   process  on $H_i$  and activate its elements. Let
	$$E=\left\{\big\{(v,i), (v,j)\big\} \in E\left( K_n^d\right) \, \left| \, v\in V\left(K_n^{d-1}\right) \,  \mbox{ and }  \,  0\leqslant i<j\leqslant g-1  \right.\right\}$$ and active all edges in $E$.
	Thus, the number of all initially activated edges is
	\[
	\sum_{i=0}^{n-1}|E_i|+ |E|= \sum_{i=0}^{n-1} m_e\left(K_n^{d-1}, r-i\right) + \binom{g}{2}n^{d-1}.
\]

Now, for  any $i \in \llbracket n  \rrbracket $, we show that all the edges incident to the vertices of $H_i$ become activated  in the $r$-edge bootstrap percolation process consecutively.
	Assume that all the edges incident to the vertices in $\bigcup_{j=0}^{i-1}V(H_j)$ became activated  in the $r$-edge bootstrap percolation process.
	Consider a vertex $x=(v, i)\in V(H_i)$. The vertex $x$ is connected to $i$ vertices $(v,0), \ldots, (v,{i-1})$ by activated edges in $E$ and also $x$ is adjacent to $r-i$ vertices in $V(H_i)$ by activated edges in $E_i$. Therefore, $x$ is incident to $r$ activated edges and so all other edges incident to $x$ in $E(H_i)$ can be activated in
	the  $r$-edge bootstrap percolation   process. Since $x$ is incident to $g-i-1$ activated edges in $E$, the number of activated edges incident to $x$ is $ i+(n-1)(d-1)+g-i-1 $
	in $E(K_n^d)$. Since $g= \max\{1, f+2\} $ and $(n-1)(d-1)+g-1=\max\big\{(n-1)(d-1), r\big\}\geqslant r$,
	all edges incident to $x$ in $K_n^d$ can be activated in the  $r$-edge bootstrap percolation  process.
	This shows that $\bigcup_{i=0}^{n-1}E_i \cup E$ is a percolating set    of desired size  in the $r$-edge bootstrap percolation  process on $K_n^d$.
\end{proof}

\begin{lemma}\label{car}
Let   $n, r, d$ be  three   positive integers  and let  $c : E(K_n^{d-1})\rightarrow\mathbbmsl{R}$ be a proper edge coloring of $K_n^{d-1}$.  Then, there is a proper edge  coloring $\widehat{c} : E( K_n^d)\rightarrow\mathbbmsl{R}$ such that
\begin{align}\label{sigma-dim}
\dim W_{\widehat{c}}\left(K_n^d, r\right)\geqslant\sum_{i=0}^{n-1}\dim W_c\left(K_n^{d-1}, r-i\right)+\binom{g}{2}n^{d-1}.
\end{align}
\end{lemma}

\begin{proof}
Consider arbitrary mutually distinct nonzero real numbers $\gamma_0, \gamma_1, \ldots, \gamma_{n-1}$  such that none of the numbers   $\gamma_i\gamma_j$ is   in the image of $c$. For every    two adjacent vertices $u=(a, i)$
and $v=(b, j)$ of  $K_{n}^{d-1}\square K_n$,  define
$$\widehat{c}(uv)=\left\{\begin{array}{ll}
c(ab)  &   \quad  \mbox{ if } i=j,\\
\vspace{-1mm}\\
\gamma_i\gamma_j   &  \quad  \mbox{ if } a=b.\end{array}\right.$$
It is straightforward  to check that  $\widehat{c} : E( K_n^d)\rightarrow\mathbbmsl{R}$ is a proper edge  coloring.
Fix  $k \in \llbracket  n \rrbracket $,  a basis  $\mathscr{B}_k $ for  $W_c(K_n^{d-1}, r-k )$,   and  a function   $\varphi\in\mathscr{B}_{k }$.  According to Definition \ref{defhh}, there exist polynomials $\{P^{\varphi}_a(x)\}_{a\in V(K_n^{d-1})}$
recognizing  $\varphi$. Define    polynomial $Q^{k ,\varphi}_u$ for any vertex $u=(a,i)\in V(K_n^{d-1}\square K_n)$
as   $Q^{k ,\varphi}_u(x)=P^{\varphi}_a(x)T^{k }_i(x),$ where
$$T^{k }_i(x)=\prod_{\ell=0}^{k -1}(\gamma_i-\gamma_\ell)\left(\frac{x}{\gamma_i}-\gamma_\ell\right).$$
Since the degree of $P^{\varphi}_a(x)$ is at most $r-k-1$  and the degree of $ T^{k }_i(x) $ is $k$, we conclude that the degree of $Q^{k ,\varphi}_u$ is at most $r-1$.
Note that $T^{k }_i(\gamma_i\gamma_j)=T^{k }_j(\gamma_i\gamma_j)$ for all   $i$ and $j$. Also,    we know from Definition \ref{defhh} that  $P^{\varphi}_a(c(ab))=P^{\varphi}_b(c(ab))$ for each edge $ab\in E(K_n^{d-1})$. Hence,     $Q^{k ,\varphi}_u$ and $Q^{k ,\varphi}_v$ have the same value on $\widehat{c}(uv)$  for any  edge  $uv\in E(K_n^{d})$. This implies that  $\{Q^{k ,\varphi}_u\}_{u\in V(K_n^d)}$   recognize   a function $\mathnormal{\Phi}_{k ,\varphi}\in W_{\widehat{c}}(K_n^d, r)$.
Now, assume that $f\geqslant 0$. Define polynomial $ R^{ s,t, y}_u $  for every integers $ s,t \in \llbracket g \rrbracket $ and  vertices  $y\in V(K_n^{d-1})$,    $u=(a, i)\in V(K_n^{d-1}\square K_n)$ as    $  R^{ s,t, y}_u(x)=S^y_a(x)L_{i}^{s,t}(x) $, where
$$S^y_a(x)=\left\{\begin{array}{ll}
0  &  \quad \mbox{ if } y\neq a,\\
\vspace{-1mm}\\
{\displaystyle\prod_{yz\in E\left(K_n^{d-1}\right)}\big(x-c(yz)\big)} &  \quad \mbox{ if } y=a.\end{array}\right.$$
and
\[
L_i^{s,t}(x)=\left\{\begin{array}{ll}
0  & \quad  \mbox{ if } i\in \llbracket g \rrbracket \setminus \{s, t\},\\
\vspace{0mm}\\
{\displaystyle\prod_{ \ell \in \llbracket g \rrbracket \setminus  \{s,t\}}
	\frac{x-\gamma_i\gamma_{\ell}}{\gamma_s \gamma_t-\gamma_i\gamma_{\ell}}}  & \quad \mbox{ if } i\in \{s, t\},\\
\vspace{0mm}\\
{\displaystyle\prod_{ \ell \in \llbracket g \rrbracket \setminus  \{s,t\}}
	\frac{(\gamma_i-\gamma_{\ell})\left(\dfrac{x}{\gamma_i}-\gamma_{\ell}\right)}{(\gamma_s- \gamma_{\ell})(\gamma_t-\gamma_{\ell})}}  & \quad \mbox{ if }  i \in \llbracket n \rrbracket \setminus \llbracket g \rrbracket.
\end{array}\right.
\]
As  $f\geqslant 0$, we find that $r-1\geqslant (n-1)(d-1)$ and, since $g= \min \{n,   \max\{1, f+2\} \}$,  the degree of every polynomial $R_u^{s, t, y}$ is
\[
(n-1)(d-1)+g-2=\min\Big\{(n-1)d-1,   \max\big\{(n-1)(d-1)-1, r-1\big\}\Big\}\leqslant r-1.
\]
Let $u=(a, i)$ and $v=(b, j)$ be two arbitrary distinct vertices in $V(K_n^{d-1}\square K_n)$.  Note that $S^y_a(c(ab))=S^y_{b }(c(ab))=0$ for all   $ab \in E(K_n^{d-1})$. This means that $R_u^{s,t, y}(\widehat{c}(uv))=R_v^{s,t, y}(\widehat{c}(uv))$ if $i=j$. Also, it is easy to check that $L_i^{s,t}(\gamma_i\gamma_j)=L_j^{s,t}(\gamma_i\gamma_j)$ for all distinct $i,j \in \llbracket n \rrbracket $.  This means that $R_u^{s,t, y}(\widehat{c}(uv))=R_v^{s,t, y}(\widehat{c}(uv))$ if $a=b$.
Therefore,  $\{R^{s ,t, y}_u\}_{u\in V(K_n^d)}$   recognize   a function $\mathnormal{\Psi}_{s ,t, y}\in W_{\widehat{c}}(K_n^d, r)$.

Since  the number of functions  $\mathnormal{\Phi}_{k ,\varphi}$ is $\sum_{i=0}^{n-1}\dim W_c(G, r-i)$ and the number of functions $ \mathnormal{\Psi}_{s ,t, y} $ is $\binom{g}{2}n^{d-1}$,  it remains  to show that all functions $\mathnormal{\Phi}_{k,\varphi}$ and $ \mathnormal{\Psi}_{s ,t, y} $ are linearly independent.
Suppose  that
\begin{equation}\label{2sigma}
\sum_{k,\varphi}\lambda_{k,\varphi}\mathnormal{\Phi}_{k,\varphi}+\sum_{s, t, y}\mu_{s,t,y}\mathnormal{\Psi}_{s,t,y}=0
\end{equation}
for some scalars   $\lambda_{k,\varphi}, \mu_{s,t,y} \in\mathbbmsl{R}$. Let $u=(a, i)$ and $v=(b, j)$ be two adjacent vertices of $K_n^{d-1}\square K_n$. If $i=j$, then $\widehat{c}(uv)=c(ab)$ and since $S^y_a(c(ab))=0$ we find from \eqref{2sigma} that
\begin{equation}\label{sigma1}
 \sum_{k,\varphi}\lambda_{k,\varphi}\mathnormal{\Phi}_{k,\varphi}(uv)=\sum_{s, t, y}\mu_{s,t,y}\mathnormal{\Psi}_{s,t,y}(uv)=0.
\end{equation}
If $a=b$, then $\widehat{c}(uv)=\gamma_i\gamma_j$ and since $T^k_i(\gamma_i\gamma_j)=0$ we find from \eqref{2sigma} that
\begin{equation}\label{sigma2}
\sum_{k,\varphi}\lambda_{k,\varphi}\mathnormal{\Phi}_{k,\varphi}(uv)=\sum_{s, t, y}\mu_{s,t,y}\mathnormal{\Psi}_{s,t,y}(uv)=0.
\end{equation}
From \eqref{sigma1} and \eqref{sigma2} we find that
$$\sum_{k,\varphi}\lambda_{k,\varphi}\mathnormal{\Phi}_{k,\varphi}=\sum_{s, t, y}\mu_{s,t,y}\mathnormal{\Psi}_{s,t,y}=0.$$

First, we show that all scalars $\lambda_{k,\varphi}$ are equal to zero.
Towards  a contradiction, suppose  that $k_0$ is the smallest number from $ \llbracket n \rrbracket $  such that $\lambda_{k_0,\varphi}\neq0$ for some $\varphi$.
Note that, for every   $i<k$, the term $\gamma_i-\gamma_i$ appears in the expression of $T^k_i$  and so    $T^k_i=0$.
This yields  that   $Q^{k,\varphi}_{u}=0$ for any   vertex $u=(a, k_0)\in V(K_n^{d-1}\square K_n)$ and integer $k>k_0$.  Thus, for any  two adjacent vertices   $u=(a, k_0)$ and $v=(b, k_0)$ of $K_n^{d-1}\square K_n$, we have
$$0=\sum_{k,\varphi}\lambda_{k ,\varphi}\mathnormal{\Phi}_{k ,\varphi}(uv)=\sum_{k ,\varphi}\lambda_{k ,\varphi}Q^{k ,\varphi}_u\big(\widehat{c}(uv)\big)= \sum_{\varphi\in \mathscr{B}_{k_0}}\lambda_{k_0,\varphi}P^{\varphi}_a\big(c(ab)\big)T^{k_0}_{k_0}\big(c(ab)\big).$$
Our assumption on  $\gamma_0, \gamma_1, \ldots, \gamma_{n-1}$   implies that  $T^{k_0}_{k_0}(c(ab))\neq0$. Therefore,
$$\left(\sum_{\varphi\in \mathscr{B}_{k_0}}\lambda_{k_0,\varphi}\varphi\right)(ab)=\sum_{\varphi\in \mathscr{B}_{k_0}}\lambda_{k_0,\varphi}P^{\varphi}_a\big(c(ab)\big)=0$$ for each edge  $ab\in E(K_n^{d-1})$. This means that
\[
\sum_{\varphi\in \mathscr{B}_{k_0}}\lambda_{k_0,\varphi}\varphi=0
\]
which is a contradiction, since $\mathscr{B}_{k_0}$ is a basis for  $W_c(G, r-k_0)$.

Next, we show that all scalars $\mu_{s, t, y}$ are equal to zero. We have  $L_i^{s,t}(\gamma_i\gamma_j)=1$ if $(i, j)= (s, t)$ and  $L_i^{s,t}(\gamma_i\gamma_j)=0$ if $(i, j)\neq (s, t)$. To verify this, note that if $i=s$, then $j\neq t$ and so the term $(\gamma_j-\gamma_j)/(\gamma_t-\gamma_j)$ appears in the expression of $L_i^{s,t}(\gamma_i\gamma_j)$, implying that  $L_i^{s,t}(\gamma_i\gamma_j)=0$. Also, for every $i, j\in \llbracket g \rrbracket  $, we have $S^y_a(\gamma_i\gamma_j)\neq 0$ if and only if $ y=a $. Therefore, for every integers  $s_0, t_0 \in  \llbracket g \rrbracket $ and vertex  $y_0\in V(K_n^{d-1})$, by letting $u=(y_0, s_0)$ and $v=(y_0, t_0)$, we have
\[
0=\left(\sum_{s, t, y}\mu_{s,t,y}\mathnormal{\Psi}_{s,t,y}\right)(uv)= \mu_{s_0, t_0, y_0}S_{y_0}^{y_0}(\gamma_{s_0}\gamma_{t_0}).
\]
Since $ S_{y_0}^{y_0}(\gamma_{s_0}\gamma_{t_0})\neq 0 $, we conclude that $ \mu_{s_0, t_0, y_0}=0 $. Hence, we proved that all scalars $\lambda_{k ,\varphi}$ and $\mu_{s, t, y}$
are zero. This completes the proof.
\end{proof}

The following theorem particularity demonstrates   that the equalities hold in \eqref{sigma-me} and \eqref{sigma-dim}.

\begin{theorem}
Let  $n, d$ be  two   positive integers. Then,
there is   a proper edge coloring $ c_{n, d} : E(K_n^d) \rightarrow \mathbbmsl{R }$
such that $m_e(K_n^d, r)= \dim W_{c_{n, d}}(K_n^d, r)$ for any nonnegative integer $ r $. Moreover,
	\begin{align}\label{return-me}
m_e\left(K_n^d, r\right)=\sum_{i=0}^{n-1} m_e\left(K_n^{d-1}, r-i\right) + \binom{g}{2}n^{d-1}.
\end{align}
\end{theorem}

\begin{proof}
We  prove     by induction on $d$ that $m_e(K_n^d, r)= \dim W_{c_{n, d}}(K_n^d, r)$ for a    proper edge coloring $ c_{n, d} : E(K_n^d) \rightarrow \mathbbmsl{R }$. Note that the latter equality trivially holds if  $d$ is replaced by $0$.
So, assume that there exists a proper edge coloring $ c_{n, d-1} : E(K_n^{d-1}) \rightarrow \mathbbmsl{R }$
such that $m_e(K_n^{d-1}, r)= \dim W_{c_{n, d-1}}(K_n^{d-1}, r)$ for all positive integers $n, r$.
By  Theorem \ref{hh},    Lemma  \ref{umeg} and Lemma   \ref{car}, there is a proper edge coloring  $c_{n, d} : E(K_n^d)\rightarrow\mathbbmsl{R}$ such that
\begin{align*}
m_e\left(K_n^d, r\right)&\geqslant \dim W_{c_{n, d}} \left(K_n^d, r\right)\\&\geqslant\sum_{i=0}^{n-1}\dim W_{c_{n, d-1}} \left(K_n^{d-1}, r-i\right)+\binom{g}{2}n^{d-1}\\&
=\sum_{i=0}^{n-1}m_e \left(K_n^{d-1}, r-i\right)+\binom{g}{2}n^{d-1}\\&
\geqslant m_e\left(K_n^d, r\right),
\end{align*}
meaning that  $m_e(K_n^d, r)= \dim W_{c_{n, d}}(K_n^d, r)$.
The `moreover'  statement is  straightforwardly valid from \eqref{sigma-me} and \eqref{sigma-dim}.
\end{proof}

Now, we are going to prove Theorem \ref{sigmaT2}. For this, we need some lemmas. At the beginning,  we establish    the following  lemma as mentioned in the introduction.

\begin{lemma}
Let  $n\geqslant 2$  and $d, r\geqslant 0$ be three integers with
$(n-1)(d-1)\leqslant r\leqslant (n-1)d$. Then,
	\begin{align}\label{extra}
m_e\left(K_n^d, r\right) = {{nd-r}\choose{d+1}}+\left(r-\frac{(n-1)d}{2}\right)n^d.
\end{align}
\end{lemma}

\begin{proof}
We proceed with the proof by induction on $d$. If $d=0$, then $r=0$ and \eqref{extra} is obviously  valid.
Let  $d\geqslant 1$ and  assume that \eqref{extra} holds for $d-1$. Since $(n-1)(d-1)\leqslant r\leqslant (n-1)d$ and $f=r-1-(n-1)(d-1)$, we find that $-1\leqslant f\leqslant n-2$ and so $ g=f+2 $.
We derive from \eqref{return-me} that
\begin{align*}
m_e\left(K_n^d, r\right)&=\sum_{i=0}^{n-1}m_e\left(K_n^{d-1}, r-i\right)+\binom{f+2}{2}n^{d-1}\\
&=\sum_{i=0}^{f+1}m_e\left(K_n^{d-1}, r-i\right)+\sum_{i=f+2}^{n-1}m_e\left(K_n^{d-1}, r-i\right)+\binom{f+2}{2}n^{d-1}\\
&=\sum_{i=0}^{f+1}\frac{(n-1)(d-1)n^{d-1}}{2}\\
&+\sum_{i=f+2}^{n-1}\Bigg(\binom{n(d-1)-(r-i)}{d}+\left((r-i)-\frac{(n-1)(d-1)}{2}\right)n^{d-1}\Bigg)\\
&+\binom{f+2}{2}n^{d-1}\\
&=\sum_{i=d}^{nd-r-1}\binom{i}{d}+\left(r-\frac{(n-1)d}{2}\right)n^d\\
&={{nd-r}\choose{d+1}}+\left(r-\frac{(n-1)d}{2}\right)n^d,
\end{align*}
where   the last equality is obtained from the  well known combinatorial identity   $\sum_{i=0}^{k}\binom{m+i}{m}=\binom{m+k+1}{m+1}$.
This establishes \eqref{extra} and   completes  the proof.
\end{proof}

We need the following combinatorial identity   in the  proof of next lemma.

\begin{proposition}\label{unknown}
For every two    integers  $m \geqslant 0$ and $k\geqslant 1$,
\begin{align*}
\sum_{i_1=0}^{k} \cdots \sum_{i_{k-1}=0}^{k}  \sum_{i_{k}=0}^{k-i_1-\cdots-i_{k-1}}\binom{m}{i_k}\binom{i_k}{i_{k-1}}\cdots \binom{i_{2}}{i_{1}}=\binom{m+k}{m}.
\end{align*}
\end{proposition}
\begin{proof}
It is well known that $\sum_{i=0}^{k}\binom{m}{i}\binom{k}{k-i}=  \binom{m+k}{k} $. By using this equality repeatedly, we have
\begin{align*}
\binom{m+k}{k}&=\sum_{i_1=0}^{k}\binom{m}{i_1}\binom{k}{k-i_1}\\
&=\sum_{i_1=0}^{k}\binom{m}{i_1}\sum_{i_2=0}^{k-i_1} \binom{i_1}{i_2}\binom{k-i_1}{k-i_1-i_2}\\
&=\sum_{i_1=0}^{k}\sum_{i_2=0}^{k}\binom{m}{i_1} \binom{i_1}{i_2}\binom{k-i_1}{k-i_1-i_2}\\
& \, \, \,   \vdots\\
&=\sum_{i_1=0}^{k} \cdots \sum_{i_{k-1}=0}^{k}  \sum_{i_{k}=0}^{k-i_1-\cdots-i_{k-1}}\binom{m}{i_1}\binom{i_1}{i_{2}}\cdots \binom{i_{k-1}}{i_{k}}\binom{k-i_1-\cdots-i_{k-1}}{k-i_1-\cdots-i_{k}}\\
&=\sum_{i_k=0}^{k} \cdots \sum_{i_{2}=0}^{k}  \sum_{i_{1}=0}^{k-i_2-\cdots-i_{k}}\binom{m}{i_1}\binom{i_1}{i_{2}}\cdots \binom{i_{k-1}}{i_{k}}\binom{k-i_1-\cdots-i_{k-1}}{k-i_1-\cdots-i_{k}}.
\end{align*}
By renaming the indices, we get that
\begin{align}
\sum_{i_1=0}^{k} \cdots \sum_{i_{k-1}=0}^{k}  \sum_{i_{k}=0}^{k-i_1-\cdots-i_{k-1}}\binom{m}{i_k}\binom{i_k}{i_{k-1}}\cdots \binom{i_{2}}{i_{1}}\binom{k-i_2-\cdots-i_{k}}{k-i_1-\cdots-i_{k}}=\binom{m+k}{k}.\label{non--zero}
\end{align}
Note that the nonzero terms in the left hand side of  \eqref{non--zero}   will be obtained whenever   $ i_1\leqslant \cdots\leqslant i_k$.
Since  $ i_1\leqslant \cdots\leqslant i_k\leqslant k-i_1-\cdots-i_{k-1}\leqslant k-(k-1)i_1  $, we get  that $i_1\leqslant 1$. If $i_1=1$, then $i_1= \cdots= i_k=1$ and     $ \binom{k-i_2-\cdots-i_{k}}{k-i_1-\cdots-i_{k}}=1 $. If $i_1=0$, then    $ \binom{k-i_2-\cdots-i_{k}}{k-i_1-\cdots-i_{k}}=1 $ again. This completes the proof.
\end{proof}

\begin{lemma} \label{lem:frac}
Let  $n\geqslant 2$  and $d, r\geqslant 0$ be three integers with  $ (n-1)(d-1)+1\leqslant r\leqslant (n-1)d$. Then,
	\begin{align}\label{frac}
\sum_{i_1=0}^{d} \cdots \sum_{i_{n-1}=0}^{d}
(r-i_1-\cdots-i_{n-1})\binom{d}{i_1}\binom{i_1}{i_2}\cdots \binom{i_{n-2}}{i_{n-1}}
	=\left(r-\frac{(n-1)d}{2}\right)n^d
\end{align}
and
\begin{align}\label{binom}	
\sum_{i_1=0}^{d} \cdots \sum_{i_{n-2}=0}^{d} \sum_{i_{n-1}=r-i_1-\cdots-i_{n-2}}^{d}
(r-i_1-\cdots-i_{n-1})\binom{d}{i_1}\binom{i_1}{i_2}\cdots \binom{i_{n-2}}{i_{n-1}}=-\binom{nd-r}{d+1}.
\end{align}
\end{lemma}

\begin{proof}
For convenience, let $A$ and $B$ be the left hand sides of \eqref{frac} and \eqref{binom}, respectively.
By substituting  $d-i_{\ell}$ with $i_{\ell}$ for $ \ell =1, \ldots, n-1$, we   may write
\begin{align}
A&=\sum_{i_1=0}^{d} \cdots \sum_{i_{n-1}=0}^{d}
\big(r-(n-1)d+i_1+\cdots+i_{n-1}\big)\binom{d}{i_1}\binom{d-i_1}{d-i_2}\cdots \binom{d-i_{n-2}}{d-i_{n-1}}\nonumber\\
&= \sum_{i_1=0}^{d} \cdots \sum_{i_{n-1}=0}^{d}
\big(r-(n-1)d+i_1+\cdots+i_{n-1}\big) \binom{d}{i_{n-1}}\binom{i_{n-1}}{i_{n-2}}\cdots \binom{i_{2}}{i_{1}}\nonumber\\
&= \sum_{i_1=0}^{d} \cdots \sum_{i_{n-1}=0}^{d}
\big(r-(n-1)d+i_1+\cdots+i_{n-1}\big) \binom{d}{i_{1}}\binom{i_{1}}{i_{2}}\cdots \binom{i_{n-2}}{i_{n-1}}\nonumber\\
&= \big(2r-(n-1)d\big)\sum_{i_1=0}^{d} \cdots \sum_{i_{n-1}=0}^{d}
\binom{d}{i_{1}}\binom{i_{1}}{i_{2}}\cdots \binom{i_{n-2}}{i_{n-1}}\nonumber\\
&- \sum_{i_1=0}^{d} \cdots \sum_{i_{n-1}=0}^{d}
(r-i_1-\cdots-i_{n-1}) \binom{d}{i_{1}}\binom{i_{1}}{i_{2}}\cdots \binom{i_{n-2}}{i_{n-1}}\nonumber\\
&=\big(2r-(n-1)d\big)n^d-A,\label{xxxyyyzzz}
\end{align}
where  the last equality comes from
\begin{align*}
\sum_{i_1=0}^{d} \cdots \sum_{i_{n-1}=0}^{d}
\binom{d}{i_{1}}\binom{i_{1}}{i_{2}}\cdots \binom{i_{n-2}}{i_{n-1}}&=
\sum_{i_1=0}^{d}\sum_{i_2=0}^{i_1} \cdots \sum_{i_{n-1}=0}^{i_{n-2}}
\binom{d}{i_{1}}\binom{i_{1}}{i_{2}}\cdots \binom{i_{n-2}}{i_{n-1}}\\
&=\sum_{i_1=0}^{d}\sum_{i_2=0}^{i_1} \cdots \sum_{i_{n-2}=0}^{i_{n-3}}
\binom{d}{i_{1}}\binom{i_{1}}{i_{2}}\cdots \binom{i_{n-3}}{i_{n-2}}\sum_{i_{n-1}=0}^{i_{n-2}}\binom{i_{n-2}}{i_{n-1}}\\
&=\sum_{i_1=0}^{d}\sum_{i_2=0}^{i_1} \cdots \sum_{i_{n-3}=0}^{i_{n-4}}
\binom{d}{i_{1}}\binom{i_{1}}{i_{2}}\cdots\binom{i_{n-4}}{i_{n-3}}\sum_{i_{n-2}=0}^{i_{n-3}} \binom{i_{n-3}}{i_{n-2}}2^{i_{n-2}}\\
& \, \, \,   \vdots\\
&=\sum_{i_1=0}^{d}\binom{d}{i_1}(n-1)^{i_1}\\
&=n^d.
\end{align*}
As    \eqref{frac}   is deduced from  \eqref{xxxyyyzzz}, we are done.

We prove \eqref{binom} by induction on $n$. If $n=2$, then $r=d$ and obviously  \eqref{binom} holds. Let  $n\geqslant 3$ and assume that  \eqref{binom} is valid for $n-1$. If $i_1\leqslant d-1$, then the nonzero terms in the left hand side of  \eqref{binom}   will be obtained whenever all indices $i_1, \ldots, i_{n-1}$ are at most $d-1$. However, $i_{n-1}\geqslant r-i_1-\cdots-i_{n-2}\geqslant(n-1)(d-1)+1-(n-2)(d-1)= d$, a contradiction. Therefore,
\begin{align}\label{key}
B&=\sum_{i_2=0}^{d} \cdots \sum_{i_{n-2}=0}^{d}\sum_{i_{n-1}=r-d-i_2-\cdots-i_{n-2}}^{d}
(r-d-i_2-\cdots-i_{n-1})\binom{d}{i_2}\binom{i_2}{i_3}\cdots \binom{i_{n-2}}{i_{n-1}}.
\end{align}
If $(n-1)(d-1)+2\leqslant r\leqslant (n-1)d$, then $(n-2)(d-1)+1\leqslant r-d\leqslant (n-2)d$ and so the right hand side of \eqref{key} is equal to $-\binom{(n-1)d-(r-d)}{d+1}=-\binom{nd-r}{d+1}$, by the  induction hypothesis.
So, it remains to consider the case $ r=(n-1)(d-1)+1 $. We find  from \eqref{key} that
\begin{align}
B&=\sum_{i_2=0}^{d} \cdots \sum_{i_{n-2}=0}^{d}\sum_{i_{n-1}=r-d+1-i_2-\cdots-i_{n-2}}^{d}
(r-d-i_2-\cdots-i_{n-1})\binom{d}{i_2}\binom{i_2}{i_3}\cdots \binom{i_{n-2}}{i_{n-1}}\nonumber\\
&=\sum_{i_2=0}^{d} \cdots \sum_{i_{n-2}=0}^{d}\sum_{i_{n-1}=r-d+1-i_2-\cdots-i_{n-2}}^{d}
(r-d+1-i_2-\cdots-i_{n-1})\binom{d}{i_2}\binom{i_2}{i_3}\cdots \binom{i_{n-2}}{i_{n-1}}\nonumber\\
&-\sum_{i_2=0}^{d} \cdots \sum_{i_{n-2}=0}^{d}\sum_{i_{n-1}=r-d+1-i_2-\cdots-i_{n-2}}^{d}
\binom{d}{i_2}\binom{i_2}{i_3}\cdots \binom{i_{n-2}}{i_{n-1}}.\label{xyz}
\end{align}
As $r-d+1=(n-2)(d-1)+1$, it  follows form  the  induction hypothesis and \eqref{xyz} that
\begin{align}
B&=-\binom{(n-1)d-\big((n-2)(d-1)+1\big)}{d+1}\nonumber\\
&-\sum_{i_2=0}^{d} \cdots \sum_{i_{n-2}=0}^{d}\sum_{i_{n-1}=r-d+1-i_2-\cdots-i_{n-2}}^{d}
\binom{d}{i_2}\binom{i_2}{i_3}\cdots \binom{i_{n-2}}{i_{n-1}}\nonumber\\
&=-\binom{d+n-3}{d+1}-\sum_{i_2=0}^{d} \cdots \sum_{i_{n-2}=0}^{d}\sum_{i_{n-1}=(n-2)(d-1)+1-i_2-\cdots-i_{n-2}}^{d}
\binom{d}{i_2}\binom{i_2}{i_3}\cdots \binom{i_{n-2}}{i_{n-1}}.\label{xxyyzz}
\end{align}
We claim that
\begin{align}\label{n-3}
\sum_{i_2=0}^{d} \cdots \sum_{i_{n-2}=0}^{d}\sum_{i_{n-1}=(n-2)(d-1)+1-i_2-\cdots-i_{n-2}}^{d}
\binom{d}{i_2}\binom{i_2}{i_3}\cdots \binom{i_{n-2}}{i_{n-1}}=\binom{d+n-3}{d}.
\end{align}
The claim trivially holds for $n=3$. So, assume that $n\geqslant4$.
If $i_2\leqslant d-1$, then the nonzero terms in the left hand side of   \eqref{n-3} will be obtained whenever all indices $i_2, \ldots, i_{n-1}$ are at most $d-1$. However, $i_{n-1}\geqslant (n-2)(d-1)+1-i_2-\cdots-i_{n-2} \geqslant  (n-2)(d-1)+1-(n-3)(d-1)= d$, a contradiction. Hence, it is enough to show that $
C= \binom{d+n-3}{d} $, where
\begin{align*}
C=\sum_{i_3=0}^{d} \cdots \sum_{i_{n-2}=0}^{d}\sum_{i_{n-1}=(n-3)(d-1)-i_3-\cdots-i_{n-2}}^{d}
\binom{d}{i_3}\binom{i_3}{i_4}\cdots \binom{i_{n-2}}{i_{n-1}}.
\end{align*}
By substituting  $d-i_{\ell}$ with $i_{\ell}$ for $ \ell =3, \ldots, n-1$,  we derive   that
\begin{align*}
C&=\sum_{i_3=0}^{d} \cdots \sum_{i_{n-2}=0}^{d}\sum_{i_{n-1}=0}^{n-3-i_3-\cdots-i_{n-2}}
\binom{d}{d-i_3}\binom{d-i_3}{d-i_4}\cdots \binom{d-i_{n-2}}{d-i_{n-1}}\\
&=\sum_{i_3=0}^{n-3} \cdots \sum_{i_{n-2}=0}^{n-3}\sum_{i_{n-1}=0}^{n-3-i_3-\cdots-i_{n-2}}
\binom{d}{i_{n-1}}\binom{i_{n-1}}{i_{n-2}}\cdots \binom{i_4}{i_3}\\
&=\binom{d+n-3}{d},
\end{align*}
where the last equality is obtained from Proposition \ref{unknown}.
This establishes the claim.
Now, it follows from \eqref{xxyyzz} and \eqref{n-3} that
$$B=-\binom{d+n-3}{d+1}-\binom{d+n-3}{d}=-\binom{d+n-2}{d+1}.$$
As $r=(n-1)(d-1)+1$, we are done.
\end{proof}

Now,  we are ready to prove our main result. Recall Theorem \ref{sigmaT2}.

\begin{reptheorem}{sigmaT2}
Let  $n\geqslant 2$  and $d, r\geqslant 0$ be three integers with $ 0\leqslant r\leqslant (n-1)d$. Then,
\begin{align}\label{sigmasigma}
	m_e\left(K_n^d, r\right) = \sum_{i_1=0}^{r-1} \sum_{i_2=0}^{r-i_1-1} \cdots \sum_{i_{n-1}=0}^{r-i_1-\cdots-i_{n-2}-1}
(r-i_1-\cdots-i_{n-1})\binom{d}{i_1}\binom{i_1}{i_2}\cdots \binom{i_{n-2}}{i_{n-1}}.
\end{align}		
\end{reptheorem}

\begin{proof}
Fix $n\geqslant 1$ and define $a_n(s, t)$ for every integers $s, t$ with $ s\geqslant 1 $ and $ 0\leqslant t\leqslant s-1 $  as follows:
\begin{align}\label{return-an}
a_n(s, t)=\left\{\begin{array}{ll}
s &  \quad  \mbox{ if } t=0.\\
\vspace{-1mm}\\
\displaystyle{\sum_{i=1}^{\min\{s-t, n-1\}}}a_{n}(s-i, t-1) & \quad   \mbox{ otherwise}.\\
\end{array}\right.
\end{align}
We prove by induction on $t$ that
\begin{align}\label{mix}
a_n(s,t)=\sum_{i=0}^{s-t-1}a_{n-1}(s-t, i)\binom{t}{i}
\end{align}
for any $n\geqslant 2$.
In view of   \eqref{return-an}, one concludes that   \eqref{mix}   holds for $t=0$. Let  $t\geqslant 1$ and assume that \eqref{mix} holds for $t-1$. Using  \eqref{return-an},
we may write
\begin{align*}
a_n(s,t)&={\sum_{i=1}^{\min\{s-t,  n-1\}}}a_{n}(s-i, t-1)\\
&=\sum_{i=1}^{\min\{s-t,  n-1\}}\sum_{j=0}^{s-t-i}a_{n-1}(s-t-i+1, j)\binom{t-1}{j}\\
&=\sum_{j=0}^{s-t-1}a_{n-1}(s-t, j)\binom{t-1}{j}+\sum_{i=2}^{\min\{s-t,  n-1\}}\sum_{j=0}^{s-t-i}a_{n-1}(s-t-i+1, j)\binom{t-1}{j}\\
&=\sum_{j=0}^{s-t-1}a_{n-1}(s-t, j)\binom{t-1}{j}+\sum_{j=0}^{s-t-2}\sum_{i=2}^{\min\{s-t-j,  n-1\}}a_{n-1}(s-t-i+1, j)\binom{t-1}{j}\\
&=\sum_{j=0}^{s-t-1}a_{n-1}(s-t, j)\binom{t-1}{j}+\sum_{j=0}^{s-t-2}\sum_{i=1}^{\min\{s-t-j-1,  n-2\}}a_{n-1}(s-t-i, j)\binom{t-1}{j}\\
&=\sum_{j=0}^{s-t-1}a_{n-1}(s-t, j)\binom{t-1}{j}+\sum_{j=0}^{s-t-2}a_{n-1}(s-t, j+1)\binom{t-1}{j}\\
&=\sum_{j=0}^{s-t-1}a_{n-1}(s-t, j)\binom{t-1}{j}+\sum_{j=0}^{s-t-1}a_{n-1}(s-t, j)\binom{t-1}{j-1}\\
&=\sum_{j=0}^{s-t-1}a_{n-1}(s-t, j)\binom{t}{j},
\end{align*}
as required.

By repeatedly using \eqref{mix}, we  get  that
\begin{align}
\sum_{i_1=0}^{r-1}a_n(r, i_1)\binom{d}{i_1}
&=\sum_{i_1=0}^{r-1}\sum_{i_2=0}^{r-i_1-1}a_{n-1}(r-i_1, i_2)\binom{d}{i_1}\binom{i_1}{i_2}\nonumber\\
& \, \, \,   \vdots\nonumber\\
&=\sum_{i_1=0}^{r-1} \sum_{i_2=0}^{r-i_1-1} \cdots \sum_{i_{n}=0}^{r-i_1-\cdots-i_{n-1}-1}
a_1(r-i_1-\cdots-i_{n-1}, i_n)\binom{d}{i_1}\binom{i_1}{i_2}\cdots \binom{i_{n-1}}{i_{n}}\nonumber\\
&=\sum_{i_1=0}^{r-1} \sum_{i_2=0}^{r-i_1-1} \cdots \sum_{i_{n-1}=0}^{r-i_1-\cdots-i_{n-2}-1}
(r-i_1-\cdots-i_{n-1})\binom{d}{i_1}\binom{i_1}{i_2}\cdots \binom{i_{n-2}}{i_{n-1}},\hspace{-5mm}\label{sigma:mex}
\end{align}
where \eqref{sigma:mex}  comes from
\begin{align*}
a_1(s, t)=\left\{\begin{array}{ll}
s &  \quad  \mbox{ if } t=0,\\
\vspace{-1mm}\\
0 &  \quad  \mbox{ otherwise}\\
\end{array}\right.
\end{align*}
which is in turn   obtained  from   \eqref{return-an}.

In order to prove \eqref{sigmasigma} and in view of \eqref{sigma:mex}, it is enough to      establish    that
\begin{align}\label{a_n}
m_e\left(K_n^d, r\right)= \sum_{j=0}^{r-1}a_n(r, j)\binom{d}{j}
\end{align}
when $0\leqslant r\leqslant (n-1)d$.
We prove \eqref{a_n} by induction on $d$.
If $d=0$, then $r=0$ and \eqref{a_n}  clearly holds. Let  $d\geqslant 1$ and assume that  \eqref{a_n} holds for $d-1$.
First, assume that $r\leqslant (n-1)(d-1)$. We have $g=1$ and so it follows from  \eqref{return-me} that
\begin{align*}
m_e\left(K_n^d, r\right)&=\sum_{i=0}^{n-1}m_e\left(K_n^{d-1}, r-i\right)\\
 &=\sum_{i=0}^{n-1}\sum_{j=0}^{r-i-1}a_{n}(r-i, j)\binom{d-1}{j}\\
   &=\sum_{j=0}^{r-1}a_{n}(r, j)\binom{d-1}{j}+\sum_{i=1}^{n-1}\sum_{j=0}^{r-i-1}a_{n}(r-i, j)\binom{d-1}{j}\\
  &=\sum_{j=0}^{r-1}a_{n}(r, j)\binom{d-1}{j}+\sum_{j=0}^{r-2}\sum_{i=1}^{\min\{r-j-1,  n-1\}}a_{n}(r-i, j)\binom{d-1}{j}\\
&=\sum_{j=0}^{r-1}a_{n}(r, j)\binom{d-1}{j}+\sum_{j=0}^{r-2}a_{n}(r, j+1)\binom{d-1}{j}\\
&=\sum_{j=0}^{r-1}a_{n}(r, j)\binom{d-1}{j}+\sum_{j=0}^{r-1}a_{n}(r, j)\binom{d-1}{j-1}\\
&=\sum_{j=0}^{r-1}a_{n}(r, j)\binom{d}{j},
\end{align*}
as required. Next, assume that $(n-1)(d-1)+1\leqslant r\leqslant (n-1)d$. In order to prove \eqref{a_n} and in view of \eqref{sigma:mex}, it suffices to show  that the right hand sides of \eqref{extra} and \eqref{sigmasigma} are equal.
Equivalently, by letting
\[
S= \sum_{i_1=0}^{r-1} \sum_{i_2=0}^{r-i_1-1} \cdots \sum_{i_{n-1}=0}^{r-i_1-\cdots-i_{n-2}-1}
(r-i_1-\cdots-i_{n-1})\binom{d}{i_1}\binom{i_1}{i_2}\cdots \binom{i_{n-2}}{i_{n-1}},
\]
we should show that  $$ S= {{nd-r}\choose{d+1}}+\left(r-\frac{(n-1)d}{2}\right)n^d. $$
We claim  that
\begin{align*}
S=&\sum_{i_1=0}^{d} \cdots \sum_{i_{n-2}=0}^{d} \sum_{i_{n-1}=0}^{r-i_1-\cdots-i_{n-2}-1}
(r-i_1-\cdots-i_{n-1})\binom{d}{i_1}\binom{i_1}{i_2}\cdots \binom{i_{n-2}}{i_{n-1}}\label{tt}.
\end{align*}
To see this,
assume that $1\leqslant k\leqslant n-2$ and $s=r-i_1-\cdots-i_{k-1}-1$. The upper bound of $k$th summation notation  in the right hand side of  \eqref{sigmasigma} is $s$. We show that $s$ can be replaced by $d$.
Note that $\binom{d}{i_1}\binom{i_1}{i_2}\cdots \binom{i_{n-2}}{i_{n-1}}=0$ if $i_k>d$. Therefore, there is nothing to prove  if $s\geqslant d$. So, assume that $s\leqslant d-1$.
Since   the upper bound of $(k+1)$th summation notation  in the right hand side of  \eqref{sigmasigma}  is $s-i_k$,
the upper bound of $k$th summation notation  can be go up to $d$ instead of $s$. This establishes the claim.

Now, it follows from  Lemma \ref{lem:frac} that
\begin{align*}
S&=\sum_{i_1=0}^{d} \cdots \sum_{i_{n-1}=0}^{d}
(r-i_1-\cdots-i_{n-1})\binom{d}{i_1}\binom{i_1}{i_2}\cdots \binom{i_{n-2}}{i_{n-1}}\\
&-\sum_{i_1=0}^{d} \cdots \sum_{i_{n-2}=0}^{d} \sum_{i_{n-1}=r-i_1-\cdots-i_{n-2}}^{d}
(r-i_1-\cdots-i_{n-1})\binom{d}{i_1}\binom{i_1}{i_2}\cdots \binom{i_{n-2}}{i_{n-1}}\\
&=\left(r-\frac{(n-1)d}{2}\right)n^d+\binom{nd-r}{d+1}.
\end{align*}
This
establishes \eqref{a_n}  and completes the proof.
\end{proof}

\begin{remark}
	The explicit  value of $m_e(K_2^d, r)$ was presented  in  \cite{ham, mor} by a rather complicated formula. A simple expression for $m_e(K_2^d, r)$ is obtained     from   \eqref{sigmasigma}. Indeed,
	$$m_e\left(K_2^d, r\right)=\sum_{i=0}^{r}(r-i)\binom{d}{i}$$
	for any integer  $r$ with $0\leqslant r \leqslant d$.
\end{remark}

Recall that the {\sl weight}  of a tuple is defined to be   the number of its
nonzero components. The following result was proved for $n=2$ in \cite{mor}. Recall Theorem \ref{assxxx}.

\begin{reptheorem}{assxxx}
Let $n\geqslant2$ and $r$ be two fixed positive  integers  and let  $d$ be an   integer  tending  to infinity. Then,
\[
  m\left(K_n^d, r\right)=\big(1+o(1)\big)\frac{d^{r-1}}{r!}.
\]
\end{reptheorem}

\begin{proof}
It follows from \eqref{sigmasigma} that $m_e(K_n^d, r)\geqslant(1+o(1))\binom{d}{r-1}$ and so    we obtain  from \eqref{m-me} that
\begin{align}\label{lb}
 m\left(K_n^d, r\right)\geqslant \frac{1+o(1)}{r}\binom{d}{r-1}.
\end{align}

Consider a family $ \mathscr{F} $ of $ r $-subsets of $ \llbracket d \rrbracket $ such that every $ (r-1) $-subset of $ \llbracket d \rrbracket $ is
contained in at least one element of $\mathscr{F} $.
Let $U\subseteq \llbracket 2  \rrbracket^d$ be the set of all characteristic vectors corresponding to    elements of $\mathscr{F}$ and let
$ W\subseteq V( K_n^d) $ be the set of  all  vertices    of weight $r-2$. We claim that $ U\cup W $ percolates
in the $ r $-neighbor bootstrap percolation process  on $ K_n^d $.

First, note that the vertices of weights $r-3,  \ldots, 1, 0$ can be respectively activated.
This is possible since every vertex of weight $r-i$ is adjacent to $(d-r+i)(n-1)\geqslant r$  vertices of weight $r-i + 1$ for $ i =3, \ldots, r$.
Next, note that the vertices of weight $r-1$ can be activated.
To see this, we show that the vertices of weight $r-1$ with $0, 1, \ldots, r-1$ components in $  \llbracket n \rrbracket\setminus  \llbracket 2  \rrbracket $ can be respectively activated. All vertices of weight $ r-1 $ whose all  components are in $ \llbracket 2  \rrbracket$ can be activated, since such vertices have $ r - 1 $ neighbors in $ W $ and at least one neighbor in $ U $. Now, for $i=1, \ldots, r-1$, every vertex of weight $ r-1 $ with $i$ components in $ \llbracket n \rrbracket\setminus  \llbracket 2  \rrbracket $ can be activated, since such vertices have $ r - 1 $ neighbors in $ W $ and at least one neighbor of weight $ r-1 $ with $i-1$ components in $ \llbracket n \rrbracket\setminus  \llbracket 2  \rrbracket $.
Finally, note that  the vertices of weights $r,  \ldots, d$ can be respectively activated. This is possible  since  every vertex of weight $r+i$ is adjacent to  $r+i$ vertices of weight $r+i-1$ for $ i =0, 1,  \ldots, d-r$. This establishes  the claim.

By a result of R\"{o}dl \cite{Rodl}, there exists such a family  $ \mathscr{F}  $ with
$ |\mathscr{F} | = \frac{1 + o(1)}{r}\binom{d}{r-1} $. Therefore,
\begin{align}\label{ub}
m\left(K_n^d, r\right)\leqslant |U|+|W|= \frac{1+o(1)}{r}\binom{d}{r-1}+(n-1)^{r-2}\binom{d}{r-2}.
\end{align}
The result follows from \eqref{lb} and \eqref{ub}.
\end{proof}

It is worth mentioning  that it remains as an open challenging problem to find       the exact formula  for  $m(K_n^d, r)$  in general cases.
We refer to see \cite{bid} for   some results on    small $r$.

\end{document}